\documentclass{amsart}
\usepackage{graphicx} 
\usepackage{amsmath,amsfonts,amssymb,amsthm,mathrsfs}
\usepackage{rotating,pgf,tikz,float}
\usepackage{hyperref,graphicx,color}
\usepackage{subcaption}
\usepackage[english]{babel}
\usepackage{imakeidx}
\usetikzlibrary{automata,positioning,arrows.meta,shapes.misc}

\usepackage[normalem]{ulem}
\usepackage{indentfirst}
\usepackage{tikz}
\usepackage{tikz-3dplot}
\usepackage{pgfplots}
\pgfplotsset{compat=1.18}
\makeindex
\usepackage[hmargin=3.0cm,bmargin=3.0cm,tmargin=3.5cm]{geometry}
\usepackage{enumerate}
\usepackage{cleveref}
\renewcommand{\autoref}[1]{\Cref{#1}}

\newcommand{\txtb}[1]{{ #1}}

\theoremstyle{plain}
\newtheorem{theorem}{Theorem}[section]
\newtheorem{lemma}[theorem]{Lemma}
\newtheorem{proposition}[theorem]{Proposition}
\newtheorem{corollary}[theorem]{Corollary}

\newtheorem*{conjecture*}{Conjecture}

\theoremstyle{definition}
\newtheorem{remark}[theorem]{Remark}
\newtheorem{definition}[theorem]{Definition}
\newtheorem{example}[theorem]{Example}

\numberwithin{equation}{section}
\numberwithin{figure}{section}

\newcommand{\Graph}{\mathcal G}
\newcommand{\HGraph}{\mathcal{H}} 
\newcommand{\EdgeSet}{\mathcal{E}} 
\newcommand{\VertexSet}{\mathcal{V}} 

\newcommand{\Partition}{\mathcal{P}}

\newcommand{\nenergy[1]}{{\Lambda}^N_{#1}} 
\newcommand{\denergy[1]}{{\Lambda}^D_{#1}} 
\newcommand{\noptenergy[1]}{{\mathcal L}^{N}_{#1}} 

\def\:{\thinspace:\thinspace}

\newcommand{\R}{\mathbb{R}}
\newcommand{\N}{\mathbb{N}}

\newcommand{\G}{\mathcal{G}}

\DeclareMathOperator{\dist}{dist}

\title[On Courant-Type Bounds via Neumann Domains on Quantum Graphs]{On Courant-type bounds and spectral partitioning via Neumann domains on quantum graphs}
\author{Lu\'is N. Baptista}
\author{Matthias Hofmann}
\address{Departamento de Ci\^encias Matem\'aticas, Faculdade de Ci\^encias da Universidade de Lisboa, 1749-016 Campo Grande, Lisboa, Portugal}
\email{
lcbaptista@ciencias.ulisboa.pt}
\address{Fakult\"at f\"ur Mathematik und Informatik, FernUniversit\"at in Hagen, 58084 Hagen, Germany}
\email{matthias.hofmann@fernuni-hagen.de}

\thanks{\emph{Acknowledgements.} The authors would like to thank James Kennedy for helpful comments. M. Hofmann was supported by the Funda\c{c}\~ao para a Ci\^encia e a Tecnologia (FCT), Portugal, within the scope of the projects ``Spectral Optimal Partitions: geometric and numerical analysis'', reference \href{https://doi.org/10.54499/2023.13921.PEX}{2023.13921.PEX}. The figures in this article were generated with the assistance of artificial intelligence tools based on the GPT-4 architecture and adjusted as necessary for this article.
}
\date{}

\begin{document}

\begin{abstract}
We study the structure of eigenfunctions of the Laplacian on quantum graphs, with a particular focus on Morse eigenfunctions via nodal and Neumann domains. Building on Courant-type arguments, we establish upper bounds for the number of nodal points and explore conditions under which Neumann domains of eigenfunctions correspond to minimizers of a class of spectral partition problems, often known as spectral minimal partitions. Our main focus is the analysis on tree graphs, where we characterize the spectral energies of such partitions and relate them to the eigenvalues of the Laplacian under genericity assumptions. Notably, we introduce a notion analogous to Courant-sharpness for Neumann counts and identify when spectral minimal partitions coincide with partitions formed by Neumann domains of eigenfunctions.
\end{abstract}

\maketitle

\section{Introduction}

The study of quantum graphs has garnered significant attention in recent years, particularly in the context of understanding the behavior of the Laplacian operator on these structures. Quantum graphs, which are networks of vertices and edges equipped with differential operators, serve as a versatile model for various physical and mathematical phenomena; see~\cite{Mug19} and~\cite{BeKu13} for background references.

One area of interest within this field is the investigation of Neumann domains (see~\cite{ABBE-ND-20}). Neumann domains describe a natural way to partition a domain by cutting through extremal points of an eigenfunction and provide an alternative viewpoint on the nodal properties of eigenfunctions. They can be seen as a counterpart of the well-studied nodal domains.

\txtb{In recent years, \emph{spectral minimal partitions} have become an essential tool for understanding the geometry of eigenfunctions. In the classical setting of bounded Euclidean domains one typically works with Dirichlet boundary conditions and partitions that minimize the maximum of the first Dirichlet eigenvalue on the subdomains; see, for instance,~\cite{CTV05, HHT09, BNHe17, BBN18}. 

On quantum graphs one can, in addition, study spectral minimal partitions for the Laplacian with standard (Neumann–Kirchhoff) vertex conditions by minimizing the maximum of the \emph{second} eigenvalue on each cluster; this Neumann-type variant has been developed in~\cite{KeKuLeMu21, HKMP21, HoKe21}. In parallel, two associated notions of domains play a central role: \emph{nodal domains}, i.e.\ the connected components of the complement of the nodal set where an eigenfunction has a fixed sign, and \emph{Neumann domains}, obtained by cutting the underlying space or graph along extremal points of the eigenfunction; see~\cite{BaBeRaSm12, BBW15, ABBE-ND-20}.}

An interesting development in this area is the discovery of interlacing estimates that relate Neumann and Dirichlet-type partitions (see~\cite{HoKe21}). These estimates, which connect the spectral properties of Neumann partitions with those of Dirichlet partitions, are closely related to the nodal partitions of eigenfunctions and offer a deeper understanding of the interplay between different types of spectral partitions.

In this paper, we explore nodal properties of eigenfunctions and the relation between Neumann spectral minimal partitions and Neumann partitions of eigenfunctions. By examining these connections, we aim to shed light on the underlying principles governing the spectral properties of quantum graphs and their applications. More precisely, we will
\begin{enumerate}
\item establish bounds on the number of nodes of eigenfunctions;
\item provide conditions under which partitions formed by Neumann domains correspond to spectral minimal partitions;
\item show that, under suitable sharpness conditions, the spectral minimal partitions coincide with the $n$th eigenvalue of the standard Laplacian.
\end{enumerate}

Importantly, we remove the commonly used genericity assumption whenever possible, proving results without relying on it -- a significant step forward, since this assumption is prevalent in the literature on nodal and Neumann domains (see, e.g.,~\cite{Ber08, ABB18, ABBE-ND-20, GSW04}). Furthermore, we uncover connections between Neumann partitions and Neumann domains, an area that is substantially more challenging than its nodal counterpart. These results not only deepen the understanding of Neumann structures but also offer optimism for extending analogous results to manifolds and domains, where such progress has been elusive.

The study of nodal properties of eigenfunctions goes back to~\cite{Stu36} in the one-dimensional case. In higher dimensions, one can prove only an estimate on the number of nodal domains, known as Courant's nodal theorem (see~\cite{Cou23}), which shows that the $n$th eigenfunction has at most $n$ nodal domains. By Pleijel's theorem it can be shown that, in higher dimensions, this inequality is strict for almost all eigenfunctions (see~\cite{Ple56}). In recent years, research on this topic has shifted to nodal properties of eigenfunctions on quantum graphs, where, due to the failure of unique continuation, only weaker bounds on the number of nodal domains hold (see~\cite{Ber08}), and a Pleijel-type theorem was developed in~\cite{HKMP21pleijel}. For tree graphs (see~\cite{Sch06}) it can be shown, under certain genericity assumptions, that the eigenfunctions each have precisely $n$ nodal domains and $n-1$ nodes. In order to prove bounds on the number of nodal domains without any genericity assumption, we need to adapt Courant's original argument.

A second line of investigation concerns spectral minimal partitions and their relation to nodal domains of eigenfunctions. These partitions have been studied both in Euclidean domains (see~\cite{CTV05, HHT09, BNL17, BNHe17, BBN18}) and recently on quantum graphs (see~\cite{KeKuLeMu21, HKMP21, HoKe21}). Neumann domains, which arise from gradient flow lines of eigenfunctions, are known to contain at least one node in each nodal component, both in domains and on graphs~\cite{AB19}. In the quantum graph case, we will show that when there is exactly one node in each Neumann domain of an eigenfunction, a connection can be made between spectral minimal partitions and nodal domains.

A third objective is to understand when spectral minimal partitions correspond to nodal domains of eigenfunctions. This correspondence holds when the partition is bipartite. However, as noted in~\cite{HHT09}, due to Pleijel’s result, most spectral minimal partitions in higher dimensions do not arise from eigenfunctions. Nevertheless, by using double coverings if necessary, one can construct solutions to the eigenvalue equation via linear combinations of their corresponding minimizers under certain matching conditions (see, e.g.,~\cite{CTV05, HHT09}). For Neumann partitions such a construction is more complicated, and additional matching conditions for spectral minimizers become necessary, since linear combinations of the corresponding eigenfunctions might not even be continuous. In the quantum graph case that we discuss, we will follow a similar strategy, introducing a set of matching conditions for the spectral minimizers that will allow the construction of eigenfunctions.

We now present the main results of this work, accompanied by commentary and examples that highlight their limitations and demonstrate their optimality. Throughout, we consider the standard Laplacian operator on a compact metric graph $\mathcal{G}$, whose eigenvalues form an ordered sequence
\[
0 = \mu_1(\mathcal{G}) < \mu_2(\mathcal{G}) \leq \mu_3(\mathcal{G}) \leq \dots,
\]
where $\mu_k(\mathcal{G})$ denotes the $k$th eigenvalue of the Laplacian on $\mathcal{G}$. Our focus will often be on the first non-zero eigenvalue, the spectral gap, denoted by $\mu_2(\mathcal{G})$. For a positive integer $k$, we denote by $\mathcal{L}^N_k(\mathcal{G})$ the minimal energy of a Neumann spectral $k$-partition, defined as
\begin{equation}\label{spm neumann intro}
\mathcal{L}^N_k(\mathcal{G}) := \inf_{\mathcal P \in \mathfrak C_k
(\G)} \max_{i=1,\ldots, k} \mu_2(\G_i)
\end{equation}

\txtb{where $\mathfrak C_k(\mathfrak G)$ denotes the set of exhaustive
$k$-partitions of $\mathcal G$ into connected, pairwise disjoint clusters
$\mathcal G_1,\ldots,\mathcal G_k$ in a cut graph of $\mathcal G$;
see Section~\ref{sec:preliminaries} for the precise definition.}

In addition, we require the notions of \emph{Neumann domains}, \emph{equipartitions}, and \emph{Morse} and \emph{generic} eigenfunctions.

\txtb{Precise definitions are given in Section~\ref{sec:preliminaries}.
\begin{itemize}
\item A Neumann domain of an eigenfunction is a connected component obtained
after cutting the graph at its Neumann points, that is, at non-nodal
interior extremal points.
\item An equipartition is a partition for which all clusters have the same
relevant spectral energy.
\item A Morse eigenfunction is one whose restriction to each edge has no
degenerate critical point.
\item A generic eigenfunction is a Morse eigenfunction associated with a
simple eigenvalue, which does not vanish at vertices and has no extremal
point at an interior vertex; see Definition~\ref{def:genericmorse}.
\end{itemize}

These informal descriptions are only meant to make the statements below
self-contained; the precise graph-theoretic definitions are collected in
Section~\ref{sec:preliminaries}.

We stress that, in this article, a \emph{nodal point} of an eigenfunction means a zero at which the eigenfunction is not identically zero in any neighbourhood. In particular, points lying on an edge where the eigenfunction vanishes identically are \emph{not} counted as nodal points (see Section~\ref{sec:preliminaries} for the precise definition).}

\medskip

First, we establish a Courant-type bound for quantum graphs:

\begin{theorem}\label{Courant-type}
Let $\mathcal{G}$ be a tree graph and $\psi$ an eigenfunction associated with $\mu_k(\mathcal{G})$. Then $\psi$ has at most $k-1$ nodal points.
\end{theorem}

This result generalizes Courant’s classical theorem to tree graphs and is sharp, as shown by Example~\ref{ex:star}. However, Example~\ref{ex:tadpole} demonstrates that the bound fails for graphs with cycles, emphasizing the necessity of restricting to trees. Unlike earlier works (e.g.,~\cite{Ber08, Sch06}), our proof does not rely on genericity assumptions.

Unlike in the nodal case, where, generically, the nodal count of the $k$th eigenfunction is $k$, the situation for Neumann domains is different. On an interval, the $k$th eigenfunction has exactly $k-1$ internal extrema and hence $k-1$ Neumann domains. Similarly, for an eigenfunction on a tree graph, the number of Neumann domains is bounded above by $k-1$ as a consequence of Theorem~\ref{Courant-type}, since each Neumann domain contains a nodal point:

\begin{corollary}\label{Neumann Courant-type}
Let $\mathcal{G}$ be a tree graph and $\psi$ an eigenfunction associated with $\mu_k(\mathcal{G})$. Then $\psi$ has at most $k-1$ Neumann domains.
\end{corollary}

\txtb{In the generic case, this estimate corresponds to the bound obtained in
\cite[Theorem~2.2(1)]{AB19}; Corollary~\ref{Neumann Courant-type}
extends it to the non-generic setting on trees.}
\medskip

Second, we provide necessary and sufficient conditions for the partition energy given by the Neumann domains of an eigenfunction to reflect the eigenvalue structure. Our correspondence relies on the sharpness of the Courant-type bound on the Neumann domains in Corollary~\ref{Neumann Courant-type}.

\begin{theorem}\label{sufficient and necessary nomain domains equipartition}
If $\psi_k$ is a Morse eigenfunction on a tree graph with exactly $k-1$ Neumann domains $\{\mathcal{H}_1,\dots,\mathcal{H}_{k-1}\}$, then
\[
\mu_k(\mathcal{G}) = \mu_2(\mathcal{H}_i) \quad \text{for all } i.
\]
\end{theorem}

This result characterizes when Neumann domains form an equipartition with respect to the second eigenvalue. Example~\ref{ex:tadpole} illustrates that this does not extend to graphs with cycles.

\medskip

Third, we characterize the interplay between Neumann domains and spectral minimal partitions. We also require a notion of \emph{spectral minimizers} of our spectral minimization problem:

\txtb{\begin{definition}
Let $\mathcal{G}$ be a graph and let $(\mathcal{G}_1, \dots, \mathcal{G}_k)$ be a $k$-spectral minimal partition of $\mathcal{G}$. The \emph{spectral minimizers} are the functions $\psi_{2,i}$ (unique up to scaling) for $i=1,\ldots, k$, each associated with the second eigenvalue $\mu_2(\mathcal{G}_i)$ of the standard Laplacian on $\mathcal{G}_i$.
\end{definition}

\begin{theorem}\label{main2}
Let $\mathcal G$ be a tree graph whose eigenfunctions are generic. Then the following statements are equivalent: 
\begin{enumerate}[(i)]
\item the Neumann domains of the $k$th eigenfunction form a $(k-1)$-spectral minimal equipartition for \eqref{spm neumann intro}; 
\item the spectral minimal $k$-partition is an equipartition and the spectral minimizers of $\mathcal L_k^N(\mathcal G)$ are generic.
\end{enumerate}
Moreover, if either of these statements holds, then
\begin{equation}\label{equation spm and mun}
\mathcal{L}^N_k(\mathcal{G}) = \mu_{k+1}(\mathcal{G}).
\end{equation}
\end{theorem}}




\begin{remark}
    The equality \eqref{equation spm and mun} characterizes the energy of the spectral minimal $k$-partition in terms of the $(k+1)$st eigenvalue. Example~\ref{ex:path} demonstrates that this characterization is sharp, whereas Example~\ref{ex:star} shows that the genericity assumption is essential. 
    
    By \cite[Theorem~2.2]{AB19}, the Neumann count is generically distributed according to a probability law which, for a certain class of graphs, has been shown to be binomial. Consequently, the conclusion of Theorem~\ref{main2}, namely that the Neumann domains of the $k$th eigenfunction form a $(k-1)$-partition, holds only with a probability that depends on the number of degree-one vertices. The interval is exceptional in this regard: in that case (see also Example~\ref{ex:path}), the conclusion holds for every eigenfunction. This phenomenon can be traced back to the fact that unless a tree graph is an interval minimizers of $\mathcal L_k^N(\mathcal G)$ do not, in general, appear to be equipartitions. This contrasts with the Dirichlet case, where eigenfunctions on tree graphs are generically Courant-sharp, and consequently every spectral minimal partition arises as the nodal partition of an eigenfunction. In the graph setting, this phenomenon was first discussed in the foundational work \cite{KeKuLeMu21}, extending the classical results for domains established in \cite{HHT09}. 
    
\end{remark}

Together, these results establish a framework linking nodal and Neumann structures to spectral optimization, removing or at least weakening genericity assumptions where possible and clarifying when such assumptions are indispensable. The accompanying examples illustrate sharpness and limitations.

Let us now summarize the structure of the article. In Section~\ref{sec:preliminaries} we introduce the eigenvalue and spectral optimal partition problem and discuss basic surgery principles. In Section~\ref{sec:mainresults} we discuss and prove our main results. We study the optimality of our results via examples in Section~\ref{sec:examples}.

\section{Preliminaries: Nodal and Neumann partitions on Metric Graphs}\label{sec:preliminaries}

\subsection{Laplacian on quantum graphs}
The present paper is strongly motivated by the setting introduced in~\cite{KeKuLeMu21, HoKe21, HKMP21} among others, which we briefly recall and summarize.

A \emph{(compact) metric graph} $\G$ is a finite disjoint union of bounded intervals $(I_e)_{e\in \EdgeSet}$ connected by gluing their endpoints. The set $\VertexSet$ of glued endpoints is referred to as the \emph{vertex set} of $\G$, and the \emph{edges} $e\in \EdgeSet$ are identified with the corresponding intervals $I_e$.

On a subgraph $\HGraph$ of $\Graph$ we define the function spaces $L^2(\mathcal H)$, $H^1(\HGraph)$ and $H_0^1(\HGraph; \VertexSet_0)$, the latter for a given set $\VertexSet_0$ of vertices in $\HGraph$. We then consider the quadratic Dirichlet form
\[
a(f):= \int_{\HGraph} |f'(x)|^2\, \mathrm dx
\]
on the domain $H^1(\HGraph)$ or $H_0^1(\HGraph; \VertexSet_0)$. In the former case, the associated operator is the Laplacian with \emph{standard} or \emph{natural} vertex conditions. A metric graph with an associated differential operator is usually referred to as a \emph{quantum graph}. Let us also recall the scalar product on $L^2(\G)$, defined as
\[
\langle f, g \rangle := \int_{\G} f \overline{g} \, dx \equiv \sum_{k=1}^{m} \int_{e_k} f(x) \overline{g(x)} \, dx.
\]
By standard results, the Laplacian on a compact metric graph has discrete, real spectrum. These eigenvalues can be characterized via the min–max principle
\begin{equation*}
\mu_k(\HGraph) = \min_{\substack{X \subset H^1(\HGraph)\\ \dim(X) =k}} \max_{f\in X \setminus \{0\}} \left \{ \dfrac{\int_{\mathcal H} |f'|^2\, \mathrm dx}{\int_{\mathcal H} |f|^2\, \mathrm dx}\right \}.
\end{equation*}

We will consider families of Laplacians and eigenvalue problems, each defined on the clusters of a partition of $\Graph$. These are closely related to the concept of cutting a graph, which, together with the concept of gluing, will be essential tools in this article.

The notion of \emph{cutting a graph} is a useful one and has a precise technical definition, as in~\cite{HoKe21}. Among other things, it has appeared frequently in the context of spectral geometry of graphs as a prototypical ``surgery principle'' (see, for example,~\cite{BKKM19} and the references therein) and was also used in~\cite{KeKuLeMu21} as the basis for defining partitions of graphs. We start with the basic definition.

\begin{definition}[Cutting graphs; Partitions on graphs]
\label{def:cut}
Let $\Graph$ and $\Graph'$ be metric graphs. Then $\Graph'$ is a \emph{cut} (or \emph{cut graph}) of $\Graph$ if
\begin{enumerate}[(i)]
\item $\Graph$ and $\Graph'$ have a common edge set, and
\item for all $v'\in \VertexSet'$ there exists $v\in \VertexSet$ such that $v' \subset v$.
\end{enumerate}
Informally, $\Graph'$ is obtained from $\Graph$ by splitting some vertices into several copies, without changing the edge set.
\end{definition}

We are now ready to define $k$-partitions on graphs.

\begin{definition}\label{def:partition}
Let $k\ge 1$ and let $\Graph=(\VertexSet, \EdgeSet)$ be a metric graph. Then:
\begin{enumerate}[(i)]
\item $\Partition:=(\Graph_1, \ldots, \Graph_k)$ is a \emph{$k$-partition} of $\Graph$ if there exists a cut $\Graph'$ such that $\Graph_1, \ldots, \Graph_k$ are connected components of $\Graph'$. We refer to the components $\Graph_1,\ldots,\Graph_k$ as \emph{clusters};
\item $\Partition=(\Graph_1, \ldots, \Graph_k)$ is an \emph{exhaustive $k$-partition} if $\Graph'= \sqcup_{i=1}^k \Graph_i$ is a cut graph of $\Graph$ and $\Graph_1, \ldots, \Graph_k$ are its connected components.
\end{enumerate}
\end{definition}

We denote the class of all $k$-partitions of $\Graph$ by $\mathfrak C_k(\Graph)$. Unless explicitly stated otherwise, we will always work with exhaustive partitions, so that any $\mathcal P=(\G_1,\dots,\G_k)\in\mathfrak C_k(\G)$ consists of connected, pairwise disjoint subgraphs $\G_i$ whose union is a cut graph of $\G$. 

We can now define spectral energies for partitions:

\begin{definition}
Consider a graph $\G$ and the set of its $k$-partitions $\mathfrak C_k(\G)$. For $\mathcal{P} \in \mathfrak C_k(\G)$, $\mathcal{P}=(\G_1,...,\G_k)$, define
\begin{equation}
\Lambda_k^N (\mathcal{P})= \max_{i=1,...,k} \mu_2 (\G_i),
\end{equation}
and
\begin{equation}\label{spm neumann}
\mathcal{L}_k^N (\G) = \min_{\mathcal{P} \in \mathfrak{C}_k(\G)} \Lambda^N (\mathcal{P}).
\end{equation}
\end{definition}

\begin{remark}
We refer to a partition $\mathcal{P}^*\in \mathfrak C_k(\G)$ attaining~\eqref{spm neumann} as a spectral minimal partition (in the Neumann sense), i.e.
\[
\Lambda_k^N (\mathcal{P}^*)=\mathcal{L}_k^N (\G).
\]
\end{remark}

\txtb{We can also consider a similar problem where instead of using the spectral gap we impose Dirichlet vertex conditions. On a subgraph $\HGraph$ of $\G$ we consider Dirichlet conditions on the (topological) boundary of $\mathcal H$ \txtb{relative to $\Graph$}, then the Dirichlet eigenvalues can be characterized via minmax principle
\begin{equation*}
\lambda_k(\HGraph) = \min_{\substack{X \subset H^1_0(\HGraph, \partial \HGraph)\\ \dim(X) =k}} \max_{f\in X \setminus \{0\}} \left \{ \dfrac{\int_{\mathcal H} |f'|^2\, \mathrm dx}{\int_{\mathcal H} |f|^2\, \mathrm dx}\right \}.
\end{equation*}
and use the first eigenvalue $\lambda_1$ to define Dirichlet-type spectral minimal partitions.} We refer to the corresponding energies by $\Lambda_k^D$ and define the corresponding optimal energy
\begin{gather*}
\mathcal L_k^D(\mathcal G) := \min_{\mathcal P \in \mathfrak C_k(\mathcal G)} \Lambda^D(\mathcal P).
\end{gather*}
We will not further investigate this quantity here, but we will use interlacing inequalities studied in~\cite{HoKe21} to obtain spectral estimates between Neumann eigenvalues and spectral minimal partition energies in the Neumann sense. For a more precise definition we refer to the introductory article~\cite{KeKuLeMu21}.

\subsection{Morse and generic eigenfunctions}

On quantum graphs, eigenfunctions can vanish identically on entire edges, especially in the presence of symmetries or special vertex conditions. This vanishing complicates the direct application of classical tools like Morse theory, which relies on the non-degeneracy of critical points. In order to adapt classical results to the quantum graph case, we will consider a non-degeneracy condition for eigenfunctions of $\G$ by introducing \emph{Morse eigenfunctions}.

\begin{definition}[\cite{ABBE-ND-20}] \label{def:genericmorse}
Let $\G$ be a nontrivial (not a loop graph) standard (equipped with the standard Laplacian) graph and $f$ an eigenfunction of $\G$.
\begin{enumerate}
\item $f$ is a \emph{Morse} eigenfunction if for each edge $e$, $f|_e$ is a Morse function, i.e.\ at no point in the interior of an edge do the first and second derivatives vanish simultaneously.
\item We call $f$ a \emph{generic} eigenfunction if it is Morse and satisfies the following properties:
\begin{itemize}
\item $f$ corresponds to a simple eigenvalue;
\item $f$ does not vanish at any vertex;
\item $f$ has no extremal points at interior vertices.
\end{itemize}
\end{enumerate}
\end{definition}

A simple criterion ensuring that an eigenfunction is Morse is given by the following lemma:

\begin{lemma}[\cite{ABBE-ND-20}]
Let $f$ be a non-constant eigenfunction. Then $f$ is Morse if and only if there exists no edge $e$ such that $f|_e \equiv 0$.
\end{lemma}

We can then define nodal and Neumann domains for Morse functions:

\begin{definition}\label{nodalpoint}
Let $f \in C(\G)$. We define the following useful concepts:
\begin{enumerate}
\item The set $N(f):=\{ x \in \G: f(x)=0 \}$ is the \emph{nodal set} of $f$, and we refer to the connected components of $\G \setminus N(f)$ as the \emph{nodal domains} of $f$;
\item A point $x\in N(f)$ is a \emph{nodal point} if and only if $f \not\equiv 0$ in any neighbourhood of $x$;
\item We denote the number of nodes of $f$ by
\[
\phi(f):=\#\{x \in \G: x \text{ is a nodal point of } f \};
\]
\item $\HGraph \subset \G$ is a nodal domain of $f$ if it is the closure of a connected component of $\G \setminus N(f)$;
\item We denote the nodal count, i.e.\ the number of nodal domains, by $z(f)$.
\end{enumerate}
\end{definition}

For Neumann points we follow~\cite[Definition 7.3]{ABBE-ND-20}. Let $V_1(\G):=\{v\in\mathcal V(\G):\deg(v)=1\}$ be the set of degree-one vertices.

\begin{definition}[\cite{ABBE-ND-20}, Definition 7.3]
Let $\psi$ be an eigenfunction. Then:
\begin{enumerate}
\item A \emph{Neumann point} of $\psi$ is a point neither in the nodal set nor located at a vertex of degree one such that its normal derivative at each incident edge is zero. We define the set of Neumann points by
\[
\mathcal{N}(\psi):=\{x \in \G \setminus (N(\psi) \cup V_1(\G)): x \text{ is an extremal point of } \psi\}.
\]
\item A \emph{Neumann domain} of $\psi$ is the closure of a connected component of $\G \setminus \mathcal{N}(\psi)$.
\end{enumerate}
\end{definition}

\subsection{Surgery principle on Neumann partitions}

In this section we discuss basic monotonicity properties of the Neumann energies under graph cuts that preserve connectivity.

\begin{lemma}\label{lem:surgery}
Let $\Partition=(\Graph_1, \ldots, \Graph_k)$ be a $k$-partition. Suppose that any partition $\Partition'=(\Graph_1', \ldots, \Graph_k')$ for which each $\Graph_i'$ is a cut of $\Graph_i$ does not disconnect $\Graph_i$ for all $i=1,\ldots, k$. Then
\begin{equation*}
\nenergy[k,p](\Partition') \le \nenergy[k,p](\Partition).
\end{equation*}
\end{lemma}

\begin{proof}
Since by assumption $\Graph_j'$ is a cut graph of $\Graph_j$ for $j=1,\ldots, k$, we have
\begin{equation*}
\mu_2(\Graph_j') \leq \mu_2(\Graph_j)
\end{equation*}
by the standard surgery principle (see~\cite[Theorem~3.4]{BKKM19}), and the statement follows immediately.
\end{proof}

\begin{remark}
Note that the Dirichlet energy follows the opposite principle: the energy of a partition $\Partition'=(\Graph_1', \ldots, \Graph_k')$ where each $\Graph_i'$ is a cut of $\Graph_i$ satisfies
\begin{equation*}
\denergy[k,p](\Partition') \ge \denergy[k,p](\Partition).
\end{equation*}
\end{remark}

An immediate consequence of Lemma~\ref{lem:surgery} is:

\begin{proposition}\label{thm:mainresultequiv}
There exists a spectral minimal $k$-partition $\Partition=(\Graph_1, \ldots, \Graph_k)$ such that each $\Graph_1, \ldots, \Graph_k$ is a tree graph. In particular,
\begin{equation}
\mathcal{L}^N_{k}(\G)=\min_{\G' \subset \G} \mathcal{L}^N_{k}(\G'),
\end{equation}
where we take the minimum over the set of cut graphs of $\G$ that are trees.
\end{proposition}

\begin{proof}
By~\cite[Thm.~4.7]{KeKuLeMu21} there exists an optimal $k$-partition $\Partition=(\Graph_1, \ldots, \Graph_k)$ of $\Graph$, which by~\cite[Corollary~3.3]{HoKe21} can be assumed exhaustive. If we consider any partition $\Partition' =( \Graph_1', \ldots, \Graph_k')$, where $\Graph_1', \ldots, \Graph_k'$ are tree graphs obtained by introducing a finite number of cuts, then by Lemma~\ref{lem:surgery} we have
\begin{equation}\label{eq:candidateformin}
\noptenergy[k,p](\Graph) =\nenergy[k,p](\Partition) \ge \nenergy[k,p](\Partition').
\end{equation}
Note that $\Graph_1' \sqcup \ldots \sqcup \Graph_k'$ is a cut graph of $\Graph$, and upon gluing the individual connected components together at $k-1$ vertices, we obtain a tree graph $\widehat \Graph$ (referred to as the minimal cut graph in~\cite[Definition~2.10]{HoKe21}). Thus
\[
\nenergy[k,p](\Partition') \ge \noptenergy[k,p](\widehat \Graph).
\]
Consider now any cut graph $\Graph'$ of $\Graph$ which is a tree graph. Then any partition of $\Graph'$ is also a partition of $\Graph$ (see Definition~\ref{def:partition}). Hence
\begin{equation}\label{eq:upperboundcut}
\noptenergy[k,p](\Graph') \ge \noptenergy[k,p](\Graph)
\end{equation}
and, since $\Graph'$ was arbitrary, combining~\eqref{eq:candidateformin} and~\eqref{eq:upperboundcut} we conclude
\begin{equation*}
\mathcal{L}^N_{k}(\G)=\min_{\G' \subset \G} \mathcal{L}^N_{k}(\G').
\end{equation*}
\end{proof}

\section{Comparison to Neumann partitions of eigenfunctions}\label{sec:mainresults}
In this section we derive our main results. First, we review a result on spectral inequalities between eigenvalues and the spectral minimal $k$-energy of a given graph.

\subsection{Results on nodal partitions of eigenfunctions}

Let us now explore the relation between nodal properties of an eigenfunction of the standard Laplacian and the position of its eigenvalue in the spectrum, proving in this context Theorem~\ref{Courant-type}.

We begin with a known result that holds for generic eigenfunctions.

\begin{lemma} \label{nodal domains of trees eigenfuntion}
Let $\G$ be a tree graph and $\psi_k$ generic. Then
\[
z(\psi_k)=k
\quad\text{and}\quad
\phi(\psi_k)=k-1.
\]
\end{lemma}

\begin{proof}
This is a special case of~\cite[Theorem~2.6]{Ber08}.
\end{proof}

We intend to explore the relation between nodal properties and the position on the spectrum. To do so, we generalize an idea of Courant (see~\cite{Cou23}). We start by proving an auxiliary result.

\begin{lemma}\label{lem:courantsargument}
Suppose $\mathcal H_1, \ldots, \mathcal H_k$ are a selection of nodal components of the eigenfunction $\psi$ associated with $\mu_k$. Then for any subgraph $\mathcal H\subset \mathcal G$ containing $\mathcal H_1, \ldots, \mathcal H_k$ we have
\begin{equation}
\lambda_k(\mathcal H) = \mu_k(\G).
\end{equation}
\end{lemma}

\begin{proof}
For simplicity we denote $\mu_k(\G)$ by $\mu_k$.
Consider the restriction $w_i=\psi|_{\HGraph_i}$ to the nodal domain $\HGraph_i$, extended by $0$ to the rest of the graph. Then, by integration by parts,
\begin{equation}\label{eq:importanteqforcourant}
\langle w_i', w_i' \rangle =\int_{\HGraph_i} |w_i'|^2=\mu_k\int_{\HGraph_i} |w_i|^2=\mu_k \langle w_i, w_i \rangle\qquad \text{ for } i=1,\ldots, k.
\end{equation}

Now consider the functions $w_1,\ldots,w_k$. Then
\begin{equation}
\left \langle \sum_{j=1}^{k} \alpha_j w_j' , w_i'\right \rangle = \mu_k \left \langle \sum_{j=1}^{k} \alpha_j w_j, w_i\right \rangle
\end{equation}
for every $i=1,\ldots,k$ and all $\alpha_1,\ldots,\alpha_k \in \R$, and $\operatorname{span}\{w_1, \ldots, w_k\}$ is a $k$-dimensional subspace of $H^1(\G)$. We infer that
\begin{equation}
\mu_{k} = \max_{f \in \operatorname{span}\{w_1, \ldots, w_k\}} \frac{\int_{\G} |f'|^2\, \mathrm dx}{\int_{\mathcal G} |f|^2\, \mathrm dx}
\ge \min_{f_1, \ldots, f_k \in H^1_0(\HGraph)} \max_{f \in \operatorname{span}\{f_1, \ldots, f_k\}} \frac{\int_{\G} |f'|^2\, \mathrm dx}{\int_{\mathcal G} |f|^2\, \mathrm dx}
\ge \lambda_k(\HGraph)
\end{equation}
with $f_1, \ldots, f_k \in H^1(\G)$ taken linearly independent.

Then, by domain monotonicity, we have
\begin{equation}\label{eq:courantargument}
\mu_k \ge \lambda_k(\HGraph) \ge \lambda_k(\G) = \mu_k
\end{equation}
and hence $\mu_k=\lambda_k(\HGraph)$, as claimed.
\end{proof}

In the previous lemma we linked the spectrum of the nodal domains to the spectrum of the graph. We now establish our first main result, stating the connection between the number of nodal domains and the spectral position of the associated eigenvalue.

\begin{proof}[Proof of Theorem \ref{Courant-type}]
\txtb{
The argument follows Courant’s original strategy: we assume that an eigenfunction $\psi$ associated with $\mu_k(\G)$ has more than $k-1$ nodes, and then derive a contradiction from the min–max principle. On a tree graph, having at least $k$ nodal points implies the existence of at least $k+1$ nodal domains. Using the acyclic structure of $\G$, we can cut at suitable nodal points in such a way that $k$ of these nodal domains lie in a single connected subgraph $\mathcal H\subset\G$, while the remaining one lies outside. On $\mathcal H$, the restrictions of $\psi$ to these $k$ nodal domains span a $k$–dimensional space on which the Rayleigh quotient is identically $\mu_k(\G)$. The min–max principle and domain monotonicity then force the $k$‑th Dirichlet eigenvalue on $\mathcal H$ to equal $\mu_k(\G)$. Finally, by introducing further cuts inside $\mathcal H$ we can construct arbitrarily many linearly independent eigenfunctions with eigenvalue $\mu_k(\G)$, contradicting the finiteness of its multiplicity. This shows that the assumption of at least $k$ nodal points was false, and hence $\psi$ has at most $k-1$ nodal points
}

We now turn to the detailed argument. Let $\mu_k$ be the $k$th eigenvalue of a tree graph $\G$ associated with $\psi$. For a contradiction, suppose there are at least $k$ nodal points and consider its nodal domains
\[
\HGraph_1,\, \HGraph_2,\ldots,\HGraph_{k+1},\ldots.
\]

Without loss of generality, we may assume $k\ge 2$; otherwise, there is nothing to show. Since $\G$ is a tree graph, there exists a nodal point such that, when a cut is introduced at this point, all but one of the resulting connected components (without loss of generality, denote the missing connected component by $\mathcal H_{k+1}$) contain all other nodal points of $\psi$.

To see this, begin by selecting any nodal point $v_1$ and perform a \emph{maximal} cut at that point, obtained by removing edges incident to that node up to the point where no further nontrivial cuts can be made. Then either one of the connected components contains all nodal points, or otherwise we may choose the nodal point furthest from $v_1$ on any connected component that contains other nodal points, denote it by $v_2$, and perform a maximal cut there instead. This process will always result in a connected component containing all remaining nodal points (see Figure~\ref{nodal examples} for an illustration). Due to the Kirchhoff condition, $\psi$ must change sign in a neighbourhood of $v_2$. Hence, another connected component of the resulting cut graph is a nodal domain of $\psi$.

\begin{figure}[H]
\centering
\begin{tikzpicture}[scale=0.9]
\draw[thick] (-2,0) -- (-3.75,0);
\draw[thick] (-5,1) -- (-3.75,0);
\draw[thick] (-5,-1) -- (-3.75,0);
\draw[thick] (-2,0) -- (-1,1);
\draw[thick] (-2,0) -- (-1,-1);
\draw[thick] (-4.375,-0.5) -- (-3.75,-1);
\draw[fill] (-2,0) circle (1.pt);
\draw[fill] (-3.75,0) circle (1.pt);
\draw[fill] (-1,-1) circle (1.pt);
\draw[fill] (-1,1) circle (1.pt);
\draw[fill] (-5,1) circle (1.pt);
\draw[fill] (-5,-1) circle (1.pt);
\draw[fill] (-3.75,-1) circle (1.pt);
\draw[fill,blue] (-3.75,0) circle (1.5pt);
\draw[fill,blue] (-2,0) circle (1.5pt);
\draw[fill,blue] (-1.5,-0.5) circle (1.5pt);
\draw[fill,blue] (-4.375,-0.5) circle (1.5pt);
\node at (-3.7,0) [anchor= north] {$v_1$};
\node at (-1.55,-0.55) [anchor= north] {$v_2$};
\node at (-4.375,-0.55) [anchor= north] {$v_3$};

\draw[thick] (-2,-2.5) -- (-3.75,-2.5);
\draw[thick] (-5,-1.5) -- (-3.75,-2.5);
\draw[thick] (-5,-3.7) -- (-4.7,-3.3);
\draw[thick] (-2,-2.5) -- (-1,-1.5);
\draw[thick] (-2,-2.5) -- (-1,-3.5);
\draw[thick] (-3.75,-2.5) -- (-4.375,-3);
\draw[thick] (-4.3,-3.3) -- (-4,-3.7);
\draw[fill] (-2,-2.5) circle (1.pt);
\draw[fill] (-3.75,-2.5) circle (1.pt);
\draw[fill] (-1,-3.5) circle (1.pt);
\draw[fill] (-1,-1.5) circle (1.pt);
\draw[fill] (-5,-1.5) circle (1.pt);
\draw[fill] (-4.375,-3)circle (1pt);
\draw[fill](-5,-3.7)circle (1pt);
\draw[fill] (-4.7,-3.3)circle (1pt);
\draw[fill] (-4,-3.7) circle (1pt);
\draw[fill] (-4.3,-3.3) circle (1pt);
\draw[fill,blue] (-3.75,-2.5) circle (1.5pt);
\draw[fill,blue] (-2,-2.5) circle (1.5pt);
\draw[fill,blue] (-1.5,-3) circle (1.5pt);
\node at (-3.7,-2.5) [anchor= north] {$v_1$};
\node at (-1.55,-3.05) [anchor= north] {$v_2$};

\draw[thick] (-7,-2.5) -- (-8.4,-2.5);
\draw[thick] (-10,-1.5) -- (-9.1,-2.3);
\draw[thick] (-10,-3.5) -- (-9.1,-2.7);

\draw[thick] (-7,-2.5) -- (-6,-1.5);
\draw[thick] (-7,-2.5) -- (-6,-3.5);
\draw[thick] (-9.375,-3) -- (-8.75,-3.5);
\draw[fill] (-7,-2.5) circle (1.pt);
\draw[fill] (-6,-3.5) circle (1.pt);
\draw[fill] (-6,-1.5) circle (1.pt);
\draw[fill] (-10,-1.5) circle (1.pt);
\draw[fill] (-10,-3.5) circle (1.pt);
\draw[fill] (-8.75,-3.5) circle (1.pt);

\draw[fill] (-8.4,-2.5) circle (1.pt);
\draw[fill] (-9.1,-2.3) circle (1.pt);
\draw[fill] (-9.1,-2.7) circle (1.pt);

\draw[fill,blue] (-7,-2.5) circle (1.5pt);
\draw[fill,blue] (-6.5,-3) circle (1.5pt);
\draw[fill,blue] (-9.4,-2.975) circle (1.5pt);

\node at (-6.55,-3.05) [anchor= north] {$v_2$};
\node at (-9.375,-3.05) [anchor= north] {$v_3$};

\draw[thick] (3,-2.5) -- (1.25,-2.5);
\draw[thick] (0,-1.5) -- (1.25,-2.5);
\draw[thick] (0,-3.5) -- (1.25,-2.5);
\draw[thick] (3,-2.5) -- (4,-1.5);
\draw[thick] (3,-2.5) -- (3.2,-2.7);
\draw[thick] (3.5,-3) -- (4,-3.5);

\draw[thick] (0.625,-3) -- (1.25,-3.5);

\draw[fill] (3,-2.5) circle (1.pt);
\draw[fill] (1.25,-2.5) circle (1.pt);
\draw[fill] (4,-3.5) circle (1.pt);
\draw[fill] (4,-1.5) circle (1.pt);
\draw[fill] (0,-1.5) circle (1.pt);
\draw[fill] (0,-3.5) circle (1.pt);
\draw[fill] (1.25,-3.5) circle (1.pt);
\draw[fill] (3.2,-2.7) circle (1.pt);

\draw[fill,blue] (1.25,-2.5) circle (1.5pt);
\draw[fill,blue] (3,-2.5) circle (1.5pt);
\draw[fill] (3.5,-3) circle (1.pt);
\draw[fill,blue] (0.625,-3) circle (1.5pt);

\node at (1.3,-2.5) [anchor= north] {$v_1$};

\node at (0.625,-3.05) [anchor= north] {$v_3$};

\end{tikzpicture}
\caption{A tree graph $\G$ (top) with the nodal points of $\psi$ shown in blue. On the left, cutting the graph at $v_1$ does not yield the desired configuration, since the remaining nodal points do not lie in a single connected component. In contrast, the other two cuts shown—through $v_3$ (center) and $v_2$ (right)—each yield a connected component that contains all remaining nodal points.}
\label{nodal examples}
\end{figure}

Consider the graph obtained by removing $\mathcal H_{k+1}$; it contains $k$ nodal domains (without loss of generality $\mathcal H_1, \ldots, \mathcal H_k$) such that each remaining nodal point is adjacent to two of these nodal domains. Let $\G'$ be a connected graph for which $\mathcal H_1,\ldots, \mathcal H_k$ form an exhaustive partition. Then, by construction, the eigenfunction is fully supported on $\G'$. By Lemma~\ref{lem:courantsargument} we have $\mu_k= \lambda_k(\G')$, and we can now construct subgraphs
\begin{equation}
\G^{(m)} = \left \{x\in \G: \dist(x, \G')\le \frac{\ell_{\min}}{m+1}\right \},
\end{equation}
for $m \in \N$, where $\ell_{\min}$ is the smallest edge length. Further, by the min–max principle,
\begin{equation}
\mu_k = \lambda_k(\G^{(m)}) \ge \lambda_k(\G) = \mu_k
\end{equation}
and we infer $\lambda_k(\G^{(m)})= \mu_k$.

If $\psi_j$ were to vanish on all edges outside of $\mathcal G'$ for some
$j \in \mathbb{N}$, then its derivatives would vanish there as well. In that
case the Kirchhoff condition at the nodal point adjacent to $\mathcal H_{k+1}$
could not be satisfied unless $\psi_j \equiv 0$ on all edges adjacent to this
nodal point. Consequently, we may argue inductively that the eigenfunction
either vanishes on each of the sets
$\mathcal H_1, \ldots, \mathcal H_{k+1}$, or else possesses an interior nodal
point on one of the sets $\mathcal H_1, \ldots, \mathcal H_k$.

We now show that the latter alternative is impossible. Without loss of
generality, suppose that $\mathcal H_1$ is a domain adjacent to a single
nodal point and that $\psi_j$ has an interior nodal point in $\mathcal H_1$.
Then the restriction of $\psi_j$ to $\mathcal H_1$ is an eigenfunction for
$\lambda_1(\mathcal H_1)$, which is simple; hence its eigenfunction cannot
have any interior nodal point—a contradiction.

On the other hand, if the eigenfunction has no nodal point at the boundary
of $\mathcal H_1$, then extending the first eigenfunction of $\mathcal H_1$ by
zero to the connected component obtained by the maximal cut at the nodal
point nearest to $\mathcal H_1$ yields an eigenfunction. This is again a
contradiction, since the first eigenvalue of a connected graph is simple and
its eigenfunction can only vanish at the Dirichlet vertex. Successively, the
same argument shows that the eigenfunction cannot have an interior nodal point
on any of the sets $\mathcal H_1, \ldots, \mathcal H_k$.

Hence, $\psi_j$ must be supported on at least one edge outside $\mathcal G'$,
and in particular the functions $\psi_j$ are linearly independent.

Then, as before, we have
\begin{equation}
\left \langle \sum_{j=1}^m \alpha_j \psi_j', \psi_i'\right \rangle = \mu_k\left \langle \sum_{j=1}^m \alpha_j \psi_j, \psi_i \right \rangle
\end{equation}
\txtb{for all } $i=1,\ldots, m$ and $\operatorname{span}\{\psi_1, \ldots, \psi_m\}$ is an $m$-dimensional subspace of $H^1(\G)$. Thus
\begin{equation}
\mu_k = \max_{f\in \operatorname{span}\{\psi_1, \ldots, \psi_m\}} \frac{\int_{\G} |f'|^2\, \mathrm dx}{\int_{\mathcal G} |f|^2\, \mathrm dx} \ge \min_{f_1, \ldots, f_m \in H^1(\G)} \max_{f \in \operatorname{span}\{f_1, \ldots, f_m\}} \frac{\int_{\G} |f'|^2\, \mathrm dx}{\int_{\mathcal G} |f|^2\, \mathrm dx}\ge \mu_m.
\end{equation}
Since $m$ was arbitrary, it follows that $\mu_{m} =\mu_k$ for all $m\ge k$, which contradicts the finiteness of the multiplicities of eigenvalues. By contradiction, we conclude that an eigenfunction associated with $\mu_k$ has at most $k-1$ nodal points.
\end{proof}

The following theorem is an application of the previous one in the special case $k=2$.

\begin{proposition}\label{one nodal point implies second eigen}
Consider a tree graph $\G$ and a Morse eigenfunction $\psi$ with exactly one nodal point. Then $\psi$ is an eigenfunction for $\mu_2(\G)$.
\end{proposition}

\begin{proof}
Denote by $v$ the unique nodal point of the Morse eigenfunction $\psi$ and by $\mu_\psi$ its associated eigenvalue.

Let $\varphi$ be an eigenfunction for $\mu_2(\G)$. By Theorem~\ref{Courant-type}, the function $\varphi$ has at most one nodal point. Since $\varphi$ changes sign, it must have exactly one nodal point by~\cite[Theorem 3]{K19}. Denote this nodal point by $u \in \G$.

We now show that the nodal points $u$ of $\varphi$ and $v$ of $\psi$ must coincide.
Assume, for a contradiction, that $u\neq v$. Since the nodal domains of $\psi$ exhaust the whole graph, one of its nodal domains $\hat \HGraph_1$ does not contain $u$ and therefore strictly contains one of the nodal domains $\hat \HGraph_2$ of $\varphi$, since only one of the nodal domains of $\varphi$ can contain $v$. The restrictions of the respective eigenfunctions are positive eigenfunctions for the corresponding Dirichlet Laplacian on $\hat \HGraph_1$ and $\hat \HGraph_2$ with the respective nodal points as Dirichlet points. They are the first eigenfunctions on the respective subgraphs, and $\hat \HGraph_2 \subsetneq \hat \HGraph_1$. By domain monotonicity of $\lambda_1$ we obtain
\[
\mu_\psi =\lambda_1(\hat \HGraph_1) < \lambda_1(\hat \HGraph_2)=\mu_2(\G) \le \mu_\psi,
\]
a contradiction.

This contradiction shows that the nodal points coincide. Therefore, the restriction of $\psi$ to each nodal domain of $\varphi$ is a positive eigenfunction and hence agrees with $\varphi$ up to a multiplicative constant. We conclude that $\mu_{\psi}= \mu_2(\G)$.
\end{proof}

\subsection{Results on Neumann partitions and spectral minimal partitions}
The following result shows how a $k$-equipartition can induce an eigenfunction that generates this partition as its Neumann domains. Here a $k$-equipartition refers to a division of the graph into $k$ connected subgraphs such that the first nontrivial eigenvalue $\mu_2$ coincides on each of the subgraphs. In this subsection we prove our remaining main results.

We start with an auxiliary statement:

\begin{proposition}\label{equipartition implies eigenfunction}
Let $\G$ be a tree graph, and let $(\HGraph_1, \ldots, \HGraph_k)$ be a partition of $\G$ such that each subgraph $\HGraph_i$ admits an eigenfunction $\psi_i$ corresponding to the same eigenvalue $\mu$ of the standard Laplacian. If each eigenfunction $\psi_i$ is generic, then there exists an eigenfunction $\psi$ of the standard Laplacian on $\G$ whose Neumann domains coincide with the partition $(\HGraph_1, \ldots, \HGraph_k)$.
\end{proposition}

\begin{proof}
Consider a partition $(\HGraph_1, \ldots, \HGraph_k)$ of the tree graph $\G$ such that each subgraph $\HGraph_i$ admits a generic eigenfunction $\psi_i$ corresponding to the same eigenvalue $\mu$.

Since $\G$ is a tree, any vertex shared between two subgraphs (a cut vertex) lies on the boundary of each subgraph it connects. In each $\HGraph_i$, such a vertex is treated as a degree-one vertex. Because the eigenfunctions $\psi_i$ are Morse, their derivatives vanish at the boundary vertices, satisfying the Neumann condition
\begin{gather*}
\partial_e \psi_i(v) = 0 \qquad \text{ for all } e\in \partial \HGraph_i.
\end{gather*}
Now fix a subgraph $\HGraph_i$ and its eigenfunction $\psi_i$. For any adjacent subgraph $\HGraph_j$ sharing a boundary vertex $v$ with $\HGraph_i$, scale $\psi_j$ by a constant $C_{i,j}\neq 0$ so that
\begin{gather*}
\psi_i (v) = C_{i,j} \psi_j(v).
\end{gather*}
This ensures continuity of the function across $v$. Since both functions have zero derivative at $v$, the Kirchhoff condition is satisfied at $v$ when the subgraphs are joined.

By repeating this process inductively—scaling and gluing eigenfunctions across shared boundary vertices—we construct a global function $\psi$ on $\G$ that is continuous and satisfies the Kirchhoff condition at every vertex. Thus, $\psi$ is an eigenfunction of the standard Laplacian on $\G$ to the eigenvalue $\mu$, and its Neumann domains coincide with the original partition $(\HGraph_1, \ldots, \HGraph_k)$.
\end{proof}

If we consider equipartitions with respect to the first nontrivial eigenvalue $\mu_2$, we can further identify the resulting number of nodal points of such a constructed eigenfunction.

\begin{corollary} \label{cor:equiimpleign2}
Let $\G$ be a tree graph, and consider an equipartition $(\HGraph_1,...,\HGraph_k)$ of $\G$ such that each subgraph $\HGraph_i$ has the same first nontrivial eigenvalue $\mu_2$ of the standard Laplacian, i.e.
\[
\mu_{2}(\HGraph_1)=\cdots=\mu_{2}(\HGraph_k),
\]
and assume that each associated eigenfunction is generic. Then there exists an eigenfunction of the standard Laplacian on $\G$ whose Neumann domains coincide with the partition $(\HGraph_1,....,\HGraph_k)$ and which has $k$ nodal points.
\end{corollary}

\begin{proof}
By assumption there exist eigenfunctions $\psi_1,\ldots, \psi_k$ on $\HGraph_1, \ldots, \HGraph_k$, each corresponding to the eigenvalue $\mu_2$, and they are generic. Then, by Proposition~\ref{equipartition implies eigenfunction}, there exists a global eigenfunction $\psi$ on $\G$ whose Neumann domains coincide with $(\HGraph_1,\ldots,\HGraph_k)$. By genericity, $\psi$ has exactly one node in each of its Neumann domains, so it has $k$ nodal points in total.
\end{proof}

Using the construction above, we can prove our second main result.

\begin{proof}[Proof of Theorem \ref{sufficient and necessary nomain domains equipartition}]
First, observe that Neumann domains of $\psi_k$ do not contain interior Neumann points by definition. Since $\mathcal{G}$ is a tree graph, Theorem~\ref{Courant-type} implies that the number of Neumann domains, each containing at least one nodal point, is not greater than $\phi(\psi_k) \leq k-1$. Suppose now that $\psi_k$ has $k-1$ Neumann domains, $\{\mathcal{H}_1, \dots, \mathcal{H}_{k-1}\}$. Then, for each $i = 1, \dots, k-1$, the restriction $\psi_k|_{\mathcal{H}_i}$ must have at least one nodal point, so $\phi(\psi_k|_{\mathcal{H}_i}) \geq 1$. Combined with the global bound, this implies
\[
\phi(\psi_k|_{\mathcal{H}_i}) = 1 \quad \text{for all } i = 1, \dots, k-1.
\]

By construction, $\mu_k$ is an eigenvalue on each $\mathcal H_i$ (see~\cite[Proposition~8.1]{ABBE-ND-20}).
Since $\psi_k|_{\mathcal{H}_i}$ is Morse by assumption and has exactly one nodal point, Proposition~\ref{one nodal point implies second eigen} implies that $\psi_k|_{\mathcal{H}_i}$ is the eigenfunction corresponding to the second eigenvalue on $\mathcal H_i$ for each $i$.

Therefore, we conclude that
\[
\mu_k(\mathcal{G}) = \mu_2(\mathcal{H}_i) \quad \text{for all } i = 1, \dots, k-1.
\]
\end{proof}

\txtb{
With this property established, we can prove our third main result, which characterizes spectral minimal partitions in terms of Neumann domains of eigenfunctions, and vice versa. 
}
\begin{proof}[Proof of Theorem \ref{main2}]
Assume $\psi_k$ is a generic eigenfunction. \txtb{Let $\mathcal{H}_1, \dots, \mathcal{H}_m$ denote the Neumann domains of $\psi_k$, where $m \le k-1$ by Corollary~\ref{Neumann Courant-type}, and define the partition
\[
\mathcal{P} = (\mathcal{H}_1, \dots, \mathcal{H}_m ).
\]}

Now suppose $m = k-1$ and that the partition $\mathcal P$ realizes the spectral minimal value, i.e.\ $\mathcal{L}^N_{k-1}(\mathcal{G}) = \Lambda^N_{k-1}(\mathcal{P})$. Then, by Theorem~\ref{sufficient and necessary nomain domains equipartition}, it follows that
\[
\mu_k(\G)=\mu_2(\mathcal{H}_1) = \dots = \mu_2(\mathcal{H}_{k-1}),
\]
and $\psi_k \big |_{\mathcal H_i}$ are generic spectral minimizers for $\mathcal P$. \txtb{In particular,
\[
\mathcal L^N_{k-1}(\mathcal{G}) = \mu_k(\G).
\]}

Conversely, suppose there exists a $(k-1)$-spectral minimal equipartition $\mathcal{H}_1, \dots, \mathcal{H}_{k-1}$ with generic spectral minimizers. Then, by Corollary~\ref{cor:equiimpleign2}, there exists an eigenfunction $\psi$ on $\mathcal{G}$ with $k-1$ nodal points, which by assumption is generic. By Lemma~\ref{nodal domains of trees eigenfuntion}, this implies that $\psi$ is the $k$th eigenfunction of~$\G$.
\end{proof}

\begin{corollary}\label{cor:specnodenontree}
Suppose there exists a cut graph $\G'$ such that each of the partition elements of the spectral minimal $(k-1)$-partition has generic spectral minimizers. Then
\[
\mathcal{L}_{k-1}(\G) \le \mu_k(\G') \leq \mu_k(\G).
\]

Suppose additionally that the optimal partition on $\G$ is a partition on $\G'$. Then
\[
\mathcal{L}_{k-1}(\G)=\mu_k(\G').
\]
\end{corollary}

\begin{proof}
Suppose $\G'$ is a tree graph which is a cut graph of $\G$. Then, by Proposition~\ref{thm:mainresultequiv} and Theorem~\ref{main2},
\[
\mathcal{L}_{k-1}(\G) \leq \mathcal{L}_{k-1}(\G') = \mu_k(\G') \leq \mu_k(\G).
\]

The second statement follows from combining Proposition~\ref{thm:mainresultequiv} and Theorem~\ref{main2}. Suppose $\mathcal P$ is an optimal $(k-1)$-partition. Then
\[
\mathcal L_{k-1}(\mathcal G')\le \Lambda(\Partition)=\mathcal{L}_{k-1}(\G)= \mathcal{L}_{k-1}(\G'),
\]
which yields the claim.
\end{proof}

\section{Examples}\label{sec:examples}

We conclude the article with several examples that demonstrate the use and limitations of our results:
\begin{itemize}
\item Example~\ref{ex:path} illustrates a direct application of Theorem~\ref{main2}, demonstrating the relationship between spectral minimal partitions and Neumann partitions. Specifically, the example shows that the spectral minimal partition constructed there corresponds to the Neumann partition of the $(k+1)$st eigenfunction.
\item Example~\ref{ex:tadpole} demonstrates the necessity of restricting our main results, Theorems~\ref{sufficient and necessary nomain domains equipartition} and~\ref{main2}, to tree graphs. In particular, Neumann partitions of eigenfunctions do not always coincide with spectral optimal partitions even if each Neumann domain contains exactly one nodal point. Furthermore, it illustrates that the bounds on the number of nodal points obtained in Theorem~\ref{Courant-type} do not extend to graphs containing cycles.
\item Example~\ref{ex:star} provides a counterexample that underscores the necessity of the genericity assumption in Theorem~\ref{main2}. It shows that without genericity of the relevant eigenfunctions, the minimal $k$-partition energy may not coincide with $\mu_{k+1}$, even if the other conditions of the statement are satisfied. Furthermore, the example addresses Theorem~\ref{Courant-type} by confirming that the upper bound on the number of nodal points remains valid even when the eigenfunctions are not generic.
\end{itemize}

\begin{example}\label{ex:path}

Consider the path graph $\mathcal I=[0,L]$ (see Figure~\ref{pathgraph}) with length $L$. Then
\[
\mu_{k+1}(\mathcal I)=\frac{k^2 \pi^2 }{L^2}.
\]
The corresponding eigenfunction is generic, and by Theorem~\ref{main2} we have
\begin{equation}
\mathcal{L}_k^N(\mathcal I) = \mu_{k+1}(\mathcal I)=\frac{k^2 \pi^2 }{L^2}.
\end{equation}
Moreover, the Neumann partition of the corresponding eigenfunction is the spectral minimal partition.

\end{example}

\begin{figure}[ht]
\centering
\begin{tikzpicture}[scale=0.9]
\draw[thick] (0.5, 0) -- (0.85, 0);
\draw[thick] (5.15, 0) -- (5.5, 0);
\foreach \x in {1,...,4} {
\draw[thick] (\x+0.15,0) -- (\x+0.85,0);
}

\node[draw, circle, fill=black, minimum size=3pt, inner sep=0pt] at (0.5 ,0) {};
\node[draw, circle, fill=black, minimum size=3pt, inner sep=0pt] at (5.5 ,0) {};
\foreach \x in {1,...,5} {
\node[draw, circle, fill=white, minimum size=3pt, inner sep=0pt] at (\x-0.15,0) {};
\node[draw, circle, fill=white, minimum size=3pt, inner sep=0pt] at (\x+0.15,0) {};
}

\foreach \x in {1,..., 5} {
\draw[dashed, thick] (\x, -0.3) -- (\x, 0.3);
}

\end{tikzpicture}
\hspace{10em}
\begin{tikzpicture}[scale=0.9]
\foreach \x in {0,...,5} {
\node[draw, circle, fill=black, minimum size=3pt, inner sep=0pt] at (\x+0.65,0) {};
\node[draw, circle, fill=black, minimum size=3pt, inner sep=0pt] at (\x+1.35,0) {};
}
\foreach \x in {0,...,5} {
\draw[thick] (\x+0.65,0) -- (\x+1.35,0);
}
\foreach \x in {1,..., 5} {
\draw[dashed, thick] (\x + 0.5, -0.3) -- (\x + 0.5, 0.3);
}

\end{tikzpicture}
\caption{Path graph with nodal domains (left) and Neumann domains (right) of the eigenfunction corresponding to $\mu_6$.}\label{pathgraph}
\end{figure}

\begin{figure}[ht]
\centering

\begin{tikzpicture}[scale=1.2]
\draw[thick] (-2,0) -- (-3.75,0);
\draw[fill] (-2,0) circle (1.5pt);
\draw[fill] (-3.75,0) circle (1.5pt);
\draw[thick] (-1.5,0) circle[radius=0.5];

\end{tikzpicture}
\hspace{10em}
\begin{tikzpicture}[scale=1.2]
\draw[thick] (1, 0) -- (4,0);
\draw[fill] (1,0) circle (1.5pt);
\draw[fill] (4,0) circle (1.5pt);
\draw[opacity=0] (1,0) circle[radius=0.5]; 
\end{tikzpicture}

\caption{Graph $\G$ (left) in Example~\ref{ex:tadpole}. The optimal Neumann partitions coincide with those on the interval of the same length (right) considered in Example~\ref{ex:path}.}
\label{loop}
\end{figure}

\begin{example}\label{ex:tadpole}
Consider the equilateral tadpole graph $\Graph$ shown in Figure~\ref{loop}, which consists of a path attached to a loop. The first Betti number is $\beta = 1$ due to the presence of a cycle. Applying \txtb{\cite[Corollary~1.3]{HoKe21}}, we obtain the inequality
\begin{equation}\label{rohlederbound}
\mathcal{L}^N_k(\Graph) \ge \mu_{k}(\Graph).
\end{equation}

However, unlike in the case of trees, this bound is not sharp.

To see this, let $\mathcal I = [0, |\Graph|]$ be an interval of the same total length. By~\cite[Theorem~4.8]{HKMP21}, we have
\[
\mathcal{L}^N_k(\Graph) \ge \frac{\pi^2 k^2}{|\Graph|^2} = \mu_{k+1}(\mathcal I),
\]
while Proposition~\ref{thm:mainresultequiv} and Example~\ref{ex:path} yield
\[
\mathcal{L}^N_k(\Graph) \le \mathcal{L}^N_k(\mathcal I) = \mu_{k+1}(\mathcal I).
\]
Combining both bounds, we find
\[
\mathcal{L}^N_k(\Graph) = \mu_{k+1}(\mathcal I).
\]
Using standard rank-perturbation arguments, we have
\begin{equation}\label{perturbationbound}
\mu_{k}(\mathcal I) < \mu_{k}(\G) \le \mu_{k+1}(\mathcal I),\qquad k \in \mathbb N
\end{equation}
which implies that $\mathcal{L}^N_k(\Graph)$ cannot equal $\mu_{k+1}(\Graph)$. Let us now investigate whether we have equality in~\eqref{rohlederbound}. We will see that the sharpness of the bound in~\eqref{rohlederbound} can fail in the presence of cycles, reinforcing the necessity of restricting to trees for optimality.

We next study the number of nodes of the corresponding eigenfunctions and, in particular, the number in each Neumann domain, in order to compare these results with those obtained in Theorems~\ref{Courant-type},~\ref{sufficient and necessary nomain domains equipartition} and~\ref{main2} for tree graphs. Suppose that each edge has length $\ell$. Using the scattering approach from~\cite{KoSm99}, the eigenvalues of $\G$ can be determined.  \txtb{We must distinguish between two cases: an eigenvalue may satisfy
\begin{gather*}
\sin(\tfrac{\sqrt{\mu}\, \ell}{2}) =0,
\end{gather*}
for $\mu >0$ and have multiplicity $2$ (see Figure~\ref{fig:tadpoleeig} for a visualization).} This first type of eigenvalues can be classified via
\begin{gather*}
\mu_{I,k} = \frac{(2k\pi)^2}{\ell^2}, \qquad k \in \mathbb N \cup \{0\}.
\end{gather*}

\begin{figure}[ht]
\centering
\begin{tikzpicture}[scale=0.5]
\begin{axis}[
title={Non-generic eigenfunction for $\mu_{I,1}$ on the tadpole graph},
xlabel={Position along graph},
ylabel={Amplitude},
legend pos=north east,
grid=both,
width=12cm,
height=6cm
]
\addplot[blue, thick, domain=0:6.28, samples=200] {sin(deg(x))};
\addlegendentry{Loop: $\sin(x)$}
\addplot[green, thick, domain=6.28:12.56, samples=2] {0};
\addlegendentry{Tail: $0$}
\addplot[red, only marks, mark=*] coordinates {(0,0) (3.14,0) (6.28,0)};
\addlegendentry{Node}
\end{axis}
\end{tikzpicture}
\hspace{10em}
\begin{tikzpicture}[scale=0.5]
\begin{axis}[
title={Generic eigenfunction for $\mu_{I,1}$ on the tadpole graph},
xlabel={Position along graph},
ylabel={Amplitude},
legend pos=north east,
grid=both,
width=12cm,
height=6cm
]
\addplot[blue, thick, domain=0:6.28, samples=200] {cos(deg(x))};
\addlegendentry{Loop}
\addplot[green, thick, domain=6.28:12.56, samples=200] {cos(deg((x-6.28)))};
\addlegendentry{Tail}
\addplot[red, only marks, mark=*] coordinates {(1.57,0) (4.71,0)};
\addplot[orange, only marks, mark=*] coordinates {(7.85,0) (10.99,0)};
\end{axis}
\end{tikzpicture}
\caption{Eigenfunctions corresponding to $\mu_{I,1}$ on $\G$ via representation as a one-dimensional domain for $\ell=2\pi$: loop from $x=0$ to $2\pi$, tail from $x=2 \pi$ to $4\pi$.}
\label{fig:tadpoleeig}
\end{figure}

The other eigenvalues can be determined by the general secular equation (see Figure~\ref{fig:tadpoleigfunction})
\begin{gather*}
\tan(\sqrt{\mu} \ell) + 2 \tan( \tfrac{\sqrt{\mu}\ell}{2}) =0
\end{gather*}
and, more importantly, the corresponding eigenfunctions are generic. We have $\mu_1(\G)=0$ and
\begin{gather*}
\mu_2(\G)< \mu_3(\G) < \mu_{I,1}
\end{gather*}
and $\mu_4(\G)=\mu_5(\G) = \mu_{I,1}$.

\begin{figure}[ht]
\centering
\begin{tikzpicture}[scale=0.5]
\begin{axis}[
xlabel={$\sqrt{\mu}$},
ylabel={Function value},
xmin=0, xmax=1,
ymin=-5, ymax=5,
legend pos=north west,
grid=major,
samples=500,
domain=0:1,
restrict y to domain=-10:10,
unbounded coords=jump,
title={Plot of the secular function},
width=12cm,
height=6cm
]
\addplot[blue, thick] {tan(deg(2*pi*x))+ 2* tan(deg(pi*x))};
\addlegendentry{$\tan(\sqrt{\mu} \ell)+ 2\tan(\sqrt{\mu}\ell/2)$}
\addplot[red, only marks, mark=*] coordinates {(0,0) (0.304, 0) (0.696, 0) (1, 0)};
\end{axis}
\end{tikzpicture}
\hspace{10em}
\begin{tikzpicture}[scale=0.5]
\begin{axis}[
title={Generic eigenfunction for $\mu_{1}(\G)$ on the tadpole graph},
xlabel={Position along graph},
ylabel={Amplitude},
legend pos=north east,
grid=both,
width=12cm,
height=6cm
]
\addplot[blue, thick, domain=0:6.28, samples=200] {1};
\addlegendentry{Loop}
\addplot[green, thick, domain=6.28:12.56, samples=200] {1};
\addlegendentry{Tail}
\end{axis}
\end{tikzpicture}\\[1em]
\begin{tikzpicture}[scale=0.5]
\begin{axis}[
title={Generic eigenfunction for $\mu_{2}(\G)$ on the tadpole graph},
xlabel={Position along graph},
ylabel={Amplitude},
legend pos=north east,
grid=both,
width=12cm,
height=6cm
]
\addplot[blue, thick, domain=0:6.28, samples=200]
{cos(0.304*deg(2*pi))*cos(0.304*deg(x-3.14))};
\addlegendentry{Loop}
\addplot[green, thick, domain=6.28:12.56, samples=200]
{cos(deg(0.304*pi))*cos(0.304*deg(12.56-x))};
\addlegendentry{Tail}
\addplot[red, only marks, mark=*] coordinates {(12.56-3.14/(0.304*2),0)};
\addlegendentry{Node}
\end{axis}
\end{tikzpicture}
\hspace{10em}
\begin{tikzpicture}[scale=0.5]
\begin{axis}[
title={Generic eigenfunction for $\mu_{3}(\G)$ on the tadpole graph},
xlabel={Position along graph},
ylabel={Amplitude},
legend pos=north east,
grid=both,
width=12cm,
height=6cm
]
\addplot[blue, thick, domain=0:6.28, samples=200]
{cos(0.696*deg(2*pi))*cos(0.696*deg(x-3.14))};
\addlegendentry{Loop}
\addplot[green, thick, domain=6.28:12.56, samples=200]
{cos(deg(0.696*pi))*cos(0.696*deg(12.56-x))};
\addlegendentry{Tail}
\addplot[red, only marks, mark=*] coordinates {(3.14-3.14/(0.696*2), 0) (3.14+3.14/(0.696*2), 0)};
\addplot[orange, only marks, mark=*] coordinates {(12.56-3.14/(0.696*2),0)};
\end{axis}
\end{tikzpicture}
\caption{Plot of the secular function and of the first three eigenfunctions via representation as a one-dimensional domain for $\ell=2\pi$: loop from $x=0$ to $2\pi$, tail from $x=2\pi$ to $4\pi$.}
\label{fig:tadpoleigfunction}
\end{figure}

The numbers of nodes of the eigenfunctions corresponding to the first five eigenvalues are $0$, $1$, $3$, $2$ and $4$, respectively. The nodes of the remaining eigenfunctions can be determined similarly. In particular, each of the generic $(4j+3)$rd eigenfunctions for $j\in \mathbb N\cup \{0\}$ has $k=4j+3$ nodes, which is greater than the bound $k-1$ for tree graphs obtained in Theorem~\ref{Courant-type}.

Turning back to the partitions formed by the Neumann domains of these eigenfunctions, we note that each Neumann domain has a unique nodal point. Nevertheless, only the Neumann domains formed by the $(4j+1)$st eigenfunctions will be spectral minimal partitions for $j\in \mathbb N \cup \{0\}$ (see also Example~\ref{ex:path}), in which case we indeed have sharpness in~\eqref{rohlederbound}. This remains true even though the eigenfunctions are generic and the spectral minimizers for each of the clusters in the spectral minimal partitions (see Example~\ref{ex:path}) are generic.

In conclusion:
\begin{itemize}
\item any bound on the number of nodal points for general graphs, as in Theorem~\ref{Courant-type}, must take into account the number of independent cycles in the graph, i.e., the first Betti number;
\item there may be no correspondence between the partitions formed by the Neumann domains of eigenfunctions and the spectral minimal partitions, even if the underlying eigenfunctions are generic and each Neumann domain has a unique nodal point.
\end{itemize}

\end{example}

\begin{example}\label{ex:star}
Consider the star graph $\mathcal S_3$ with edge lengths $1,1$ and $1+\epsilon$, where $0\le \epsilon \le 1$.
\begin{figure}[H]
\centering
\begin{tikzpicture}[scale=0.9]
\draw[thick] (-2,0) -- (-3.75,0);
\draw[thick] (-2,0) -- (-1,1);
\draw[thick] (-2,0) -- (-1,-1); 
\draw[fill] (-2,0) circle (1.5pt);
\draw[fill] (-3.75,0) circle (1.5pt);
\draw[fill] (-1,-1) circle (1.5pt);
\draw[fill] (-1,1) circle (1.5pt);
\node at (-2.75,0) [anchor= south] {$1+\epsilon$};
\node at (-1.5,0.5) [anchor= west] {$1$};
\node at (-1.5,-0.5) [anchor= west] {$1$};
\end{tikzpicture} 
\caption{Graph $\G$ of Example~\ref{ex:star}.}
\label{fig:graphstar}
\end{figure}

Let us first consider the equilateral case $\epsilon=0$. Then its eigenvalues are known (see \cite[Example~3]{Fri05}) to be
\[
\mu_{3j+1}= \pi^2 j^2,\qquad 
\mu_{3j+2}=\mu_{3j+3}= \pi^2 \bigl(j+\tfrac{1}{2}\bigr)^2,
\qquad j \in \mathbb N_0.
\]
These eigenvalues have multiplicities, and the corresponding eigenfunctions are not generic. In particular, some eigenfunctions are not fully supported: they vanish on entire edges and have zeros only at finitely many points. The fully supported eigenfunctions \txtb{in the eigenspace of $\mu_{3j+2} = \mu_{3j+3}$} have $3j+3$ nodal domains but only $3j+1$ nodal points \txtb{in agreement with Theorem~\ref{Courant-type}}.

\begin{center}
\begin{figure}[ht]
\begin{tikzpicture}[scale=0.9]
\coordinate (c) at (0,0);
\foreach \i in {1,2,3} {
  \coordinate (v\i) at (120*\i:0.5);
  \coordinate (u\i) at (120*\i:0.7);
  \coordinate (w\i) at (120*\i:1.7);
  \coordinate (x\i) at (120*\i:1.9);
  \coordinate (y\i) at (120*\i:2.4);
  \coordinate (f\i) at ($(120*\i:.6)+({120*\i+90}:.3)$);
  \coordinate (g\i) at ($(120*\i:.6)+({120*\i-90}:.3)$);
  \coordinate (d\i) at ($(120*\i:1.8)+({120*\i+90}:.3)$);
  \coordinate (e\i) at ($(120*\i:1.8)+({120*\i-90}:.3)$);
  \draw[thick] (c) -- (v\i) (u\i)--(w\i) (x\i)--(y\i);
  \draw[dashed, thick] (f\i) -- (g\i) (d\i) -- (e\i);
  \draw[fill] (y\i) circle (1.75pt);
  \draw[fill=white] (v\i) circle (1.75pt) (u\i) circle (1.75pt) (w\i) circle (1.75pt) (x\i) circle (1.75pt);
}
\draw[fill] (c) circle (1.75pt);
\end{tikzpicture}
\hspace{10em} 
\begin{tikzpicture}[scale=0.9]
\coordinate (c) at (0,0);
\foreach \i in {1,2,3} {
  \coordinate (v\i) at (120*\i:.15);
  \coordinate (w\i) at (120*\i:.95);
  \coordinate (u\i) at (120*\i:1.15);
  \coordinate (x\i) at (120*\i:1.95);
  \coordinate (y\i) at (120*\i:2.15);
  \coordinate (z\i) at (120*\i:2.55);
  \coordinate (h\i) at ({60+120*\i}:.4);
  \coordinate (f\i) at ($(120*\i:1.05)+({120*\i+90}:.3)$);
  \coordinate (g\i) at ($(120*\i:1.05)+({120*\i-90}:.3)$);
  \coordinate (d\i) at ($(120*\i:2.05)+({120*\i+90}:.3)$);
  \coordinate (e\i) at ($(120*\i:2.05)+({120*\i-90}:.3)$);
  \draw[thick] (v\i) -- (w\i) (u\i) -- (x\i) (y\i) -- (z\i);
  \draw[dashed, thick] (f\i) -- (g\i) (d\i) -- (e\i);
  \draw[thick,dashed] (c) -- (h\i);
  \draw[fill] (z\i) circle (1.75pt);
  \draw[fill=white] (v\i) circle (1.75pt) (u\i) circle (1.75pt) (w\i) circle (1.75pt) (x\i) circle (1.75pt) (y\i) circle (1.75pt);
}
\end{tikzpicture}
\vspace{.5em}
\caption{Visual representation of nodal domains for Morse eigenfunctions corresponding to \txtb{$\mu_7$ and $\mu_8=\mu_9$}. White vertices denote nodal points.}
\label{fig:nodal-3star}
\end{figure}
\end{center}

The spectral minimal Neumann partition energies for $\mathcal S_3$ are given by (see \cite[Lemma~7.4]{HKMP21})
    \begin{equation}
        \mathcal L_{3j+1}^N(\mathcal S_3) = \pi^2 \left (j+\frac{1}{2}\right )^2, \qquad \mathcal L_{3j+2}^N(\mathcal S_3) =\mathcal L_{3j+3}^N(\mathcal S_3)=\pi^2 (j+1)^2
    \end{equation}
    for $j\in \mathbb N_0$. For generic eigenfunctions corresponding to $\mu_{3j+1}$, the Neumann partition coincides with the spectral minimal partition and forms an optimal $3j$-partition, as guaranteed by Theorem~\ref{main2}. For $\mu_{3j+2}= \mu_{3j+3}$, there exist Morse eigenfunctions whose Neumann partitions correspond to optimal $3j+1$-partitions.  

    \begin{figure}[ht]
\centering
\begin{tikzpicture}[scale=0.9]
\coordinate (c) at (0,0);
\foreach \i in {1,2,3} {
  \coordinate (v\i) at (120*\i:0.4);
  \coordinate (u\i) at (120*\i:0.6);
  \coordinate (w\i) at (120*\i:1.4);
  \coordinate (x\i) at (120*\i:1.6);
  \coordinate (y\i) at (120*\i:2.4);
  \coordinate (f\i) at ($(120*\i:.5)+({120*\i+90}:.3)$);
  \coordinate (g\i) at ($(120*\i:.5)+({120*\i-90}:.3)$);
  \coordinate (d\i) at ($(120*\i:1.5)+({120*\i+90}:.3)$);
  \coordinate (e\i) at ($(120*\i:1.5)+({120*\i-90}:.3)$);
  \draw[thick] (c) -- (v\i) (u\i)--(w\i) (x\i)--(y\i);
  \draw[dashed, thick] (f\i) -- (g\i) (d\i) -- (e\i);
  \draw[fill] (y\i) circle (1.75pt);
  \draw[fill] (v\i) circle (1.75pt) (u\i) circle (1.75pt) (w\i) circle (1.75pt) (x\i) circle (1.75pt);
}
\end{tikzpicture}
\hspace{10em}
\begin{tikzpicture}[scale=0.9]
\coordinate (c) at (0,0);
\foreach \i in {1,2,3} {
  \coordinate (v\i) at (120*\i:.15);
  \coordinate (w\i) at (120*\i:.81);
  \coordinate (u\i) at (120*\i:1.01);
  \coordinate (x\i) at (120*\i:1.67);
  \coordinate (y\i) at (120*\i:1.87);
  \coordinate (z\i) at (120*\i:2.53);
  \coordinate (h\i) at ({60+120*\i}:.4);
  \coordinate (f\i) at ($(120*\i:.91)+({120*\i+90}:.3)$);
  \coordinate (g\i) at ($(120*\i:.91)+({120*\i-90}:.3)$);
  \coordinate (d\i) at ($(120*\i:1.77)+({120*\i+90}:.3)$);
  \coordinate (e\i) at ($(120*\i:1.77)+({120*\i-90}:.3)$);
  \draw[thick] (v\i) -- (w\i) (u\i) -- (x\i) (y\i) -- (z\i);
  \draw[dashed, thick] (f\i) -- (g\i) (d\i) -- (e\i);
  \draw[thick,dashed] (c) -- (h\i);
  \draw[fill] (z\i) circle (1.75pt);
  \draw[fill] (v\i) circle (1.75pt) (u\i) circle (1.75pt) (w\i) circle (1.75pt) (x\i) circle (1.75pt) (y\i) circle (1.75pt);
}
\end{tikzpicture}
\caption{Neumann domains of the Morse eigenfunctions corresponding to $\mu_7=\mu_8$ and $\mu_9$.}
\label{fig:neumann-3star}
\end{figure}

    \txtb{
    The eigenfunctions to $\mu_{3j+1}$ are generic and by Theorem~\ref{main2} since each of its Neumann domains contains precisely one node, it corresponds to an optimal $3j$-partition for $j\ge 1$. For the nonsimple eigenvalues $\mu_{3j+2}=\mu_{3j+3}$ for each $j\in \mathbb N_0$ there exists a Morse eigenfunction whose Neumann partition corresponds to an optimal $3j+1$-partition.
    
     However, there will not exist a Neumann partition of an eigenfunction that corresponds to an optimal $3j+2$-partition. In fact, there does not even exist an exhaustive $3j+2$-partitions for $j\in \mathbb N_0$, which is an equipartition. Let us demonstrate this for $j=0$ with the general case being similar. Then $\mu_2(\mathcal S_3)=\mu_3(\mathcal S_3)$ which generates only one Neumann domain and $$\mathcal{L}^N_1(\mathcal S_3)=\mu_2(\mathcal S_3)=\mu_3(\G).$$ However, $$\mathcal{L}^N_2(\mathcal S_3) = \mu_4(\mathcal S_3),$$ and by standard domain monotonicity argument this implies that the spectral minimal $2$-partition consists of an interval of length $1$ and length $2$, which is not an equipartition.

Next, consider the perturbed case $0<\epsilon< 1$. Then $\mu_2(\mathcal S_3), \mu_3(\mathcal S_3), \mu_4(\mathcal S_3)$ become simple eigenvalues. We have as before $\mathcal L^N_1(\mathcal S_3) = \mu_2(\mathcal S_3)$.   There does not exist a Morse eigenfunction associated to $\mu_3(\mathcal S_3)$ and $\mu_4(\mathcal S_3)$ that consists of two Neumann domains, for which one has two nodal points.  Let us now see, whether a $2$-equipartition exists in this case. The possible configurations for the spectral minimal $2$-partition are shown in Figure~\ref{fig:graphstar2part}:
\begin{figure}[H]
\centering
\begin{tikzpicture}[scale=0.9]
\draw[thick] (-2.75,0) -- (-4.25,0);
\draw[thick] (-2,0) -- (-1,1);
\draw[thick] (-2,0) -- (-1,-1);
\draw[thick] (-2,0) -- (-2.2,0);
\draw[dashed] (-2.5,-0.2) -- (-2.5, 0.2);
\draw[fill] (-2.2,0) circle (1.5pt);
\draw[fill] (-2,0) circle (1.5pt);
\draw[fill] (-2.75,0) circle (1.5pt);
\draw[fill] (-4.25,0) circle (1.5pt);
\draw[fill] (-1,-1) circle (1.5pt);
\draw[fill] (-1,1) circle (1.5pt);

\draw[thick] (2.2,0) -- (3.75,0);
\draw[thick] (4.2,0) -- (4.77,1);
\draw[thick] (4.2,0) -- (4.77,-1);
\draw[fill] (2.2,0) circle (1.5pt);

\draw[fill] (4.2,0) circle (1.5pt);
\draw[fill] (3.75,0) circle (1.5pt);
\draw[fill] (4.77,1) circle (1.5pt);
\draw[fill] (4.77,-1) circle (1.5pt);
\draw[dashed] (4,-0.2) -- (4, 0.2);
\draw[thick] (8.2,0) -- (9.75,0);
\draw[thick] (9.75,0) -- (10.77,1);
\draw[thick] (10,-0.2) -- (10.77,-1);
\draw[fill] (8.2,0) circle (1.5pt);
\draw[dashed] (10,0) -- (9.75, -0.2);
\draw[fill] (10,-0.2) circle (1.5pt);
\draw[fill] (9.75,0) circle (1.5pt);
\draw[fill] (10.77,1) circle (1.5pt);
\draw[fill] (10.77,-1) circle (1.5 pt);

\end{tikzpicture}
\caption{Different setting of $2-$partitions of the graph $\G$.}
\label{fig:graphstar2part}
\end{figure}

Let us describe the different configurations as shown in Figure~\ref{fig:graphstar2part}:
\begin{itemize}
\item (Left) One cell is an interval with length $2$ with a pendant edge with length $0 <\delta <1$ and the other is an interval with length $1+\epsilon-\delta$. The value of $\mu_2$ of the second cell is strictly greater than the first one.
\item (Center) The partition consists of two intervals: one of length $2$ and the other of length $1+\epsilon$. Again, the value of $\mu_2$ of the second cell is greater than the first one.
\item (Right) Similarly, any other cut will result in an interval with length smaller or equal $1$, where the value of $\mu_2$ of the second cell is greater than the first one.
\end{itemize}
None of these configurations yield equipartitions. In particular, the optimal partition (given as in Figure~\ref{fig:graphstar2part} (center)) is not an equipartition. 

If  $\epsilon=1$, there exists an equipartition consisting of two intervals with length $2$, which will be an equipartition. Suprisingly, one has $\mathcal L_2^N(\mathcal S_3) = \mu_3(\mathcal S_3)$ in this case despite the non-genericity of the eigenvalue, but there does not exist a corresponding Neumann partition of an eigenfunction, which provides the spectral minimal partition.   

The failure of the correspondence between Neumann domains of eigenfunctions and the spectral minimal $3$-partition is rooted in the fact that the eigenfunction corresponding to $\mu_3(\mathcal S_3)$ are not generic. Consequently, \autoref{main2} does not apply. This example thus highlights that genericity is essential for the correspondence between Neumann partitions of eigenfunctions and spectral minimal partitions.
}

This example shows that, even on a tree, non-generic eigenfunctions may lead to Neumann partitions whose energies do not coincide with the value of the spectral minimal partition functional predicted by Theorem~\ref{main2}. In particular, without the genericity assumption, the minimal Neumann spectral $k$-partition energy need not equal $\mu_{k+1}(\G)$, although the Courant-type bound on nodal points from Theorem~\ref{Courant-type} still holds.
\end{example}

\end{document}